\documentclass[11pt]{amsart}
\usepackage[english]{babel}
\usepackage{amssymb}
\usepackage{amsmath}
\usepackage{srcltx}
\DeclareMathOperator{\Hom}{Hom}  
  
\newcommand{\ignore}[1]{}

\newtheorem{lemma}{Lemma}
\newtheorem{theorem}{Theorem}
\newtheorem{proposition}{Proposition}
\newtheorem{corollary}{Corollary}

\theoremstyle{definition}
\newtheorem{definition}{Definition}

\newtheorem{example}{Example}

\newtheorem{remark}{Remark}



\title{A closer look at Witt rings for forms of higher degree}

\author{S. Pumpl\"un}

\email{susanne.pumpluen@nottingham.ac.uk}
\address{School of Mathematics\\
 University of Nottingham\\
 University Park\\
 Nottingham NG7 2RD\\
 United Kingdom }

\begin{document}

\maketitle

\begin{abstract}
  Witt rings for nondegenerate forms $\varphi$ of  degree $d$ over a field
of characteristic 0 or greater than $d$ were defined by Harrison and Pareigis in 1988. We revisit and discuss their definition in light of the developments in the theory of forms of higher degree that were made over the past years. We classify the $H$-forms employed in their definition, and propose to define Witt rings of diagonal forms of degree $d$.  We investigate some special cases; two different types of Witt rings for nondegenerate forms $\varphi$ of  degree $d$.
\end{abstract}

\section*{Introduction}

Let $k$ be a field of characteristic $0$ or greater than $d$.
Let $P_d(k)$ be the set of isomorphism classes of nondegenerate  forms of degree $d$ over $k$.
Together with the direct sum and the tensor product, $P_d(k)$ is a commutative semiring
with unit element  $\langle 1\rangle$.  The Grothendieck ring $ \widehat{W}_d(k)$
associated to $P_d(k)$ is called the {\it Witt-Grothendieck ring of forms of degree $d$}
over $k$, and is a commutative $k$-algebra with unit element $\langle 1\rangle$.
 The Witt-Grothendiek ring of nondegenerate forms of degree $d\geq 3$ over $k$ was defined by Harrison \cite{H}. Subsequently, Harrison and Pareigis constructed Witt rings of nondegenerate forms of higher degree $d$  over $k$  \cite{H-P}.
These rings depend on a fixed finite subgroup $H$ of $k^\times/k^{\times d}$.
Analogues of hyperbolic spaces
depending on $H$, called {\it $H$-spaces}, are  nondegenerate forms $\varphi$
of degree $d$ satisfying $a\varphi\cong\varphi$ for all $a\in H$. With the help of these $H$-spaces
an equivalence relation $\sim_H$ is imposed on  $P_d(k)$, and the resulting commutative unital ring $P_d(k)/\sim_H$
 is called the {\it Witt ring} $W_d(k,H)$.
Its addition is given by the orthogonal sum and its multiplication by the tensor product.
The Witt rings $W_d(k,H)$ thus obtained display some nice functorial properties.

Since Harrison and Pareigis defined the Witt ring $W_d(k,H)$ of higher degree forms, the theory of forms of higher degree has seen some exciting developments, e.g. see \cite{HLYZ, HLHZ, K1, K2, OR1, Pu5, Sax, SOR1, SOR2},  which for instance allow us to classify the $H$-forms used in their definition. 
 We now know that the centre algebra of almost every form is $k$, which means that almost all higher degree forms are absolutely indecomposable
 \cite{HLYZ, HLHZ}. A promising approach thus seems to be to focus on a Witt ring of diagonal forms
(``it is obvious that diagonalizable
forms enjoy nice algebro-geometric properties. Specifically, they are smooth and this
imposes very strong restriction, namely the semisimplicity, on their centers'' \cite{HLYZ}), and we will discuss this possibility.

 Our second goal is to define a  Witt ring for nondegenerate forms of higher degree $d$ over a field $k$ which, for $d=2$,
is the classical Witt ring of quadratic forms. The problem here is how to define
``hyperbolic spaces of higher degree''. The well-known situation in the degree 2 case is a very simplified one:
the ideal generated by the hyperbolic quadratic spaces in the Witt-Grothendiek ring $\widehat{W}(k)$ is exactly
the subgroup generated by hyperbolic planes. Quadratic hyperbolic spaces are extremely well-behaved: they remain hyperbolic
under arbitrary base field extensions (which allows us to define the canonical base change homomorphism
$W(k)\to W(l)$ for field extensions $l/k$), the transfer of a hyperbolic space again is a hyperbolic space
(which allows us to define  $s_*:W(l)\rightarrow W(k)$ for any non-zero $k$-linear map $s:l\rightarrow k$), and they are obviously
all even-dimensional (which allows us to canonically define
${\rm dim}:W(k)\rightarrow \mathbb{Z}$ and to study the fundamental ideal). Moreover, their discriminant is always 1.
Ideally, one would want hyperbolic forms of higher degree which satisfy all - or at least some -
of the above properties, provided they can be translated in a sensible way to the setting of higher degree forms.
 This, however, seems to be hard to achieve and we might have to settle for several different ``Witt rings'' now
instead of simply one. We look into several different ways to define an equivalence relation on the
semiring $P_d(k)$ of isomorphism classes of nondegenerate forms of degree $d$ over $k$ to obtain such rings.

 We begin by studying the different $H$-forms possible and more closely investigate $W_d(k,H )$ for the two special cases that $H=k^\times/k^{\times d}$ is maximal and that $d$ is even and $H=\{1,-1\}$ in Section \ref{sec:H-forms}.
  In Section \ref{sec:Witt ring}, we extend the definition of
$H$-forms given in \cite{H-P} to include the case that $H$ is an infinite subgroup. This leads to the definition of a Witt semiring  $W_d(k,H )$.
Only forms $\varphi$ with finite index $[H:G_k(\varphi)]$ have an additive inverse in $W_d(k,H )$.
In Section \ref{sec:4}, we define Witt semirings of diagonal forms and invariants of diagonal forms and conclude that
 the definition of Witt rings using $H$-equivalence  seems to give good results for
certain fields, if we only study diagonal cubic forms. We conclude with some brief observations on the Witt ring of separable forms and invariants of separable forms defined in \cite{R}
in Section \ref{sec:5}, before moving on to propose several variations of a Witt ring definition in Section \ref{sec:I-forms} and their pros and cons.

The situation is rather complex. For instance,
if one tries to define ``hyperbolic'' forms of degree $d\geq 3$ as a certain class of forms $\varphi$ which satisfy
 $G_l(\varphi)=\{a\in
l^\times\mid\varphi \cong a\varphi\}=l^\times$ for all field extensions $l$ of $k$, then
the tensor product of a hyperbolic form with any other form
is round and universal for all field extensions of $k$ and the sum of a hyperbolic
form with any other is still universal for all these field extensions. However, then the trace of a hyperbolic form $\varphi$
over $l$ is
not necessarily hyperbolic.
And, while we can construct plenty of these forms, there are no  diagonal forms of dimension greater than 1,
which satisfy $G_l(\varphi)=l^\times$ for all field extensions $l$ \cite{Pu1}.

If we relax this requirement and ask  the ``hyperbolic'' forms to only satisfy $G_k(\varphi)=k^\times$, then,
 if $k^\times/k^{\times d}$ is finite, there are diagonal round forms which
satisfy $G_k(\varphi)=k^\times$, and we can use these or a subclass of them to define
hyperbolic {\it diagonal} forms of higher degree and correspondingly, a Witt ring.
 (For infinite groups $k^\times/k^{\times d}$,
there are again no {\it diagonal} forms which satisfy $G(\varphi)=k^\times$.)
This leads us back to the definition by Harrison and Pareigis: For
 $H=k^\times/k^{\times d}$, the $H$-forms  are exactly the round and universal forms over $k$.

 Nonetheless, there seem to be several disadvantages to Harrison and Pareigis's approach. The most striking perhaps is the fact that
in case  $d=2$, the classical Witt ring $W(k)$
of quadratic forms over $k$ in general does not coincide with any Witt ring of the kind $W_2(k,H)$.
 More precisely, there is a canonical surjective
ring homomorphism $W(k)\to W_2(k,H)$ which has as its kernel the annihilator
${\rm Ann}_{W(k)}\{\langle 1\rangle- \langle h\rangle \,|\, h\in H\}$, because
 hyperbolic quadratic forms are always quadratic $H$-forms. However,
so are certain anisotropic quadratic forms.

When dealing with forms of degree $d$, we will always consider nondegenerate forms over fields of  characteristic 0 or $> d$.

\section{Preliminaries}

\subsection{} A $d$-{\it linear form} over $k$ is a $k$-multilinear map $\theta : V \times
\dots \times V \rightarrow k$ ($d$-copies) on a finite-dimensional vector space $V$ over $k$ which is {\it symmetric};
 i.e., $\theta (v_1,\dots, v_d)$ is invariant under all permutations of its variables.
A {\it form of degree $d$} over $k$ is a map $\varphi:V\rightarrow k$  on a finite-dimensional vector space $V$ over $k$
 such that $\varphi(a v)=a^d\varphi(v)$ for all $a\in k$, $v\in V$ and such that the map $\theta : V \times
\dots \times V \rightarrow k$,
 $$\theta(v_1,\dots,v_d)=\frac{1}{d!} \sum_{1\leq i_1< \dots<i_l\leq d}(-1)^{d-l}\varphi(v_{i_1}+ \dots +v_{i_l})$$
($1\leq l\leq d$) is a $d$-linear form over $k$. The {\it dimension} of $\varphi$ is defined as ${\rm dim}\,\varphi={\rm dim}\, V$.
By fixing a basis $\{e_1,\dots,e_n\}$ of $V$ any form $\varphi:V\rightarrow k$ of degree $d$ can be viewed as a homogeneous polynomial
of degree $d$ in $n={\rm dim}\, V$ variables $x_1,\dots,x_n$
 via $\varphi(x_1,\dots,x_n)=\varphi(x_1e_1+\dots+x_ne_n)$.
Conversely, any homogeneous polynomial of degree $d$ over $k$ is a form of degree $d$ over $k$.
Any $d$-linear form $\theta : V \times \dots \times V \rightarrow k$ induces a form $\varphi: V\rightarrow k$ of degree $d$ via
$\varphi (v)=\theta(v,\dots,v)$.
We can therefore identify $d$-linear forms and forms of degree $d$.

\subsection{}
Two $d$-linear spaces $(V_i,\theta_i)$, $i=1,2$, are called {\it isomorphic} (written
$(V_1,\theta_1)\cong (V_2,\theta_2)$ or just $\theta_1\cong\theta_2$), if there exists a bijective linear map
$f:V_1\rightarrow V_2$ such that $\theta_2(f(v_1),\dots,f(v_d))=\theta_1(v_1,\dots,v_d)$ for all $v_1,\dots,v_d\in V_1.$
A $d$-linear space $(V,\theta)$ (or $\theta$ itself) is called {\it nondegenerate}, if $v = 0$ is the only vector such
that $\theta (v, v_{2}, \dots, v_d) = 0$ for all $v_i \in V$.  $\theta$ is
nondegenerate if and only if $\theta (v, w, \dots, w) = 0$ for all $w \in V$ implies $v=0$ \cite[p.~125]{H}.
 A form  $\varphi$ of degree $d$ is called {\it nondegenerate}, if its associated $d$-linear form is  nondegenerate.
 $\varphi$ is \emph{isotropic} if $\varphi(x)=0$ for some nonzero element $x\in V$, otherwise it is called
\emph{anisotropic}.
 From now on we will exclusively deal with nondegenerate forms.

\subsection{}
The {\it orthogonal sum} $(V_1,\theta_1)\perp (V_2,\theta_2)$ of two $d$-linear spaces $(V_i,\theta_i)$, $i=1,2$, is defined
to be the $k$-vector space $V_1\oplus V_2$ together with the $d$-linear form
$$(\theta_1 \perp\theta_2)(u_1+v_1,\dots,u_d+v_d)=\theta_1(u_1,\dots,u_d)+\theta_2(v_1,\dots,v_d)$$
($u_i\in V_1$, $v_i\in V_2$).
The {\it tensor product} $(V_1,\theta_1)\otimes (V_2,\theta_2)$ is defined to be the $k$-vector
space $V_1\otimes V_2$ together with the $d$-linear form
$$(\theta_1 \otimes \theta_2)(u_1\otimes v_1,\dots,u_d\otimes v_d)=\theta_1(u_1,\dots,u_d)\cdot\theta_2(v_1,\dots,v_d)$$
 \cite{H-P}. A $d$-linear space $(V,\theta)$ is
called {\it decomposable} if $(V,\theta)\cong (V_1,\theta_1)\perp (V_2,\theta_2)$ for two nonzero $d$-linear spaces
$(V_i,\theta_i)$, $i=1,2$. A $d$-linear space $(V,\theta)$ is called {\it indecomposable} if it is not
decomposable. We distinguish between indecomposable
ones and {\it absolutely indecomposable} ones; i.e., $d$-linear spaces which stay indecomposable under each
algebraic field extension.

There exists a Krull-Schmidt Theorem for nondegenerate $d$-linear spaces, if $d \geq 3$: let $(V, \theta)$ be a nondegenerate
$d$-linear space over $k$. Then $(V, \theta)$ has a decomposition as a direct sum of nondegenerate indecomposable
$d$-linear spaces which is unique up to order and isomorphism \cite[2.3]{H}.

 If we can write $\varphi$ as $a_1x_1^d+\ldots +a_mx_m^d$, we use the notation $\varphi=\langle  a_1,\ldots
,a_n\rangle  $ and call $\varphi$ {\it diagonal}. One-dimensional forms are easy to analyze, because
$\langle a\rangle\cong \langle b\rangle$ if and only if $b=as^d$ for some $s\in k^\times$.
We have $\varphi=\langle  a_1,\ldots ,a_n\rangle    \cong \langle  a_1s_1^d,\ldots ,a_ns_n^d\rangle  $
for arbitrary $s_i\in k^{\times}$.
 Let $d\geq 3$ and suppose that $a_i,b_j\in k^\times$. Then
$$\langle  a_1,\ldots ,a_n\rangle\cong \langle  b_1,\ldots ,b_n\rangle$$
 if and only if there is a permutation $\pi\in S_n$ such that
 $\langle b_i \rangle\cong \langle a_{\pi (i)}\rangle$ for every $i$.
 A form is called {\it separable} if it becomes diagonal over the algebraic closure of the base field $k$  \cite{H}.

 \subsection{}  Let $(V,\varphi)$ be a form over $k$ of degree $d$ in $n$ variables over $k$.
 An element $a\in k$ is {\it represented by} $\varphi$ if there is an $x\in V$ such that $\varphi(x)=a$.
 An element $a\in k^{\times}$ such that $\varphi\cong a\varphi$ is called a {\it similarity factor} of the form $\varphi$.
Write $D_k(\varphi)= \{a \in k^\times\mid\varphi(x)=a \text{ for some } x\in V\}$ for the set of non-zero
elements represented by a form $\varphi$ of degree $d$ on $V$ over $k$ and $G_k(\varphi)=\{a\in
k^\times\mid\varphi \cong a\varphi\}$ for the group of similarity factors of $\varphi$ over $k$.
The subscript $k$ is omitted if it is
clear that $\varphi$ is a form over the base field $k$. If $D_k(\varphi)=k^\times$ then
$\varphi$ is called {\it universal}. If $D_l(\varphi)=l^\times$ for each field extension $l$ over $k$,
$\varphi$ is called {\it strongly universal}.
 (Each nondegenerate isotropic quadratic form over a field of characteristic
not $2$ is strongly universal.)
 A form $\varphi$  is called \emph{round} if $D(\varphi)\subset G(\varphi)$.

Note that $k^{\times d}$ is a subgroup of $ G(\varphi)$ and that $G(\varphi)=k^\times$
if $\varphi$ is the zero form. Thus, $G(\varphi)$ is a union of cosets
of $k^\times$ modulo $k^\times/k^{\times d}$. Therefore, we may also think of $G(\varphi)$ as a group of
classes of $d$th powers; that is, we can work instead with the factor group $G(\varphi)/k^{\times d}$.
By abuse of notation, we sometimes also regard $D(\varphi)$ as a subset of $k^\times/k^{\times d}$ instead of writing
$D(\varphi)/k^{\times d}$.

\subsection{} Let $(V,\varphi)$ be a form of degree $d$ over $k$.
For any field extension $l/k$, we can canonically provide $V{\otimes_k}l$ with a form of degree $d$ denoted by $\varphi_l$. A
form $\varphi$ of degree $d$ over $l$ is \emph{defined} over $k$, if there is a form $\varphi_0$ of
degree $d$ over $k$ such that $\varphi\cong\varphi_0\otimes_k l$.

Let $l/k$ be a finite field extension and $s:l\rightarrow k$ a non-zero $k$-linear map. If $\Gamma:V\times\dots\times V\rightarrow
l$ is a nondegenerate $d$-linear form over $l$, then $s\Gamma:V\times\dots\times V\rightarrow
k$ is a nondegenerate $d$-linear form over $k$, with $V$ viewed as a $k$-vector space. The $d$-linear
space $(V,s\Gamma)$ is called the
{\it transfer} of $(V,\Gamma)$, it is also denoted by $s_*(V,\Gamma)$ or $s_*(\Gamma)$.
If the map $s$ is the trace of a field extension $l/k$, we  simply write $tr_{l/k}(\Gamma)$.
We have $s_*((V_1,\theta_1)\perp (V_2,\theta_2))=s_*(V_1,\theta_1)\perp s_*(V_2,\theta_2)$,
and $s_*(V_1,\Gamma_1)\cong s_*(V_2,\Gamma_2)$ if $(V_1,\Gamma_1)\cong (V_2,\Gamma_2)$.
If $\theta$ is a $d$-linear space over $k$, then there exists a canonical isometry
$$s_*(\theta_l\otimes_l \Gamma)\cong \theta \otimes_k s_*(\Gamma).$$

\subsection{} Let $P_d(k)$ be the set of isomorphism classes of nondegenerate  forms of degree $d$ over $k$.
Together with direct sum and  tensor product, $P_d(k)$ is a commutative semi-ring
with unit element $\langle 1\rangle$. Let $ \widehat{W}_d(k)$ be the Grothendieck ring
associated to $P_d(k)$, called the {\it Witt-Grothendieck ring of forms of degree $d$}
over $k$  \cite{R}. $\widehat{W}_d(k)$ is a commutative $k$-algebra with unit element $\langle 1\rangle$.
Recall that the elements of $ \widehat{W}_d(k)$ are equivalence classes of formal differences $[\theta] -[\Gamma]$
of suitable nondegenerate forms over $k$. Two $d$-linear forms  $\theta$ and $\Gamma$ are {\it equivalent} over $k$,
if there are two  $d$-linear forms $\theta'$ and $\Gamma'$ over $k$ such that there is an isomorphism
$\theta\oplus \theta'\cong \Gamma\oplus \Gamma'$ of $d$-linear forms.
The map $\theta\mapsto [\theta - 0]$ induces a canonical embedding $i:P_d(k)\rightarrow   \widehat{W}_d(k)$.

Let $l$ be a field extension of $k$. The map
$$r^*:\widehat{W}_d(k)\rightarrow  \widehat{W}_d(l),\quad (V,\theta)\mapsto  (V,\theta)\otimes l$$
is a ring homomorphism.
Let $l$ be a finite field extension of $k$ and $s:l\rightarrow  k$ a non-zero $k$-linear map. For each non-zero $s\in \Hom_k (l,k)$, the map
$$s_*: \widehat{W}_d(l)\rightarrow  \widehat{W}_d(k),\quad (U,\Gamma)\mapsto  s(V,\Gamma)$$
is a homomorphism of $\widehat{W}_d(k)$-modules.
The composite map
$$s_*r^*:\widehat{W}_d(k)\rightarrow  \widehat{W}_d(l)\rightarrow \widehat{W}_d(k)$$
is given by multiplication with $s_*\langle 1\rangle$, since
$s_*(\theta_l)\cong \theta \otimes_k s_*\langle 1\rangle.$
The image of $s_*$ is an ideal in $\widehat{W}_d(k)$,
which is independent of the choice of the non-zero $k$-linear map $s\in \Hom_k(l,k)$.
(For the proof see \cite[5.7, p.~48]{S}, which carries over to arbitrary $d$.)
By computing $s_*(\langle 1\rangle)$, we obtain results on $s_*r^*$ and thus about $r^*$.
For instance, the kernel of $r^*$ must be contained in the annihilator of
$s_*\langle 1\rangle$.
Now the composite
$\widehat{W}_d(l)^G\rightarrow  \widehat{W}_d(k) \rightarrow  \widehat{W}_d(l)^G$ of the two maps $r^*: \widehat{W}_d(k)\rightarrow  \widehat{W}_d(l)^G$
 and $tr_{l/k}: \widehat{W}_d(l)^G\rightarrow  \widehat{W}_d(k)$ maps a $d$-linear space $(V,\theta)$ over $l$ to
$tr_{l/k}(V,\theta)$ and then to
$$tr_{l/k}(V,\theta)\otimes_k l\cong
 \oplus_{\sigma\in G}(V^\sigma,\theta^\sigma)\cong \oplus_{\sigma\in G}(V,\theta)=[l:k]\times (V,\theta)$$
\cite{R}. This yields the following result (for $d=2$, cf. \cite[8.2, p.~62]{S}):

The cokernel of the map
$r^*:  \widehat{W}_d(k) \rightarrow  \widehat{W}_d(l)^G$
and the kernel of the map
$tr_{l/k}: \widehat{W}_d(l)^G\rightarrow  \widehat{W}_d(k)$
are annihilated by the degree $[l:k]$.

 For a perfect field, the Witt-Grothendieck ring $\widehat{W}_d(k)$ is the union of the images
of the trace maps $tr_{l/k}:\widehat{W}_d(l) \rightarrow  \widehat{W}_d(k)$ over all finite field extensions $l$ of $k$.
This follows from the classification of indecomposable forms over perfect fields
 as traces of absolutely indecomposable ones \cite{Pu3}.

For nondegenerate quadratic forms, the dimension of a form induces a homomorphism $e_0:W(k)\rightarrow \mathbb{Z}_2$ called the {\it dimension index}.
Its kernel $I$, the {\it fundamental ideal} of the Witt ring, and the filtration induced
 by its powers, link the classical
Witt ring to Milnor $K$-theory and Galois cohomology. For the fundamental ideal $I$  in the classical
Witt ring of quadratic forms, the discriminant $d$ of a quadratic form  \cite[2.2.1]{S}.
induces a group homomorphism $d:I\rightarrow k^\times/k^{\times 2}\cong H^1(K,\mu_2)$. The induced map
$d:I/I^2\rightarrow k^\times/k^{\times 2}\cong H^1(K,\mu_2)$ is an isomorphism and also denoted by $e_1$ \cite{S}.
The question if there are cohomological invariants for separable forms of higher degree (i.e., of forms which become
diagonal over the algebraic closure of $k$), which generalize
the maps $e_0$ and $e_1$, was studied by Rupprecht \cite{R}. He
 found cohomological invariants for separable forms of degree $d\geq 3$ and achieved a cohomological classification of the separable forms of degree $d\geq 3$.
 Rupprecht
 defined the Witt ring $ \widehat{W}_r^{sep}(k)/H$ of separable forms of prime degree $r\not=2$ with the ideal
$H\subset  \widehat{W}_r^{sep}(k)$ in the Witt-Grothendieck ring of separable forms generated by the universal
form $h_r=\langle 1\rangle \oplus (l, tr_{l/k}(\langle x\rangle))$, where $l=k[x]/(\Phi)$ is a separable $k$-algebra of
dimension $r-1$ over $k$ and where $\Phi(x)=x^{r-1}+\dots+x+1$ is the $r$-th cyclotomic polynomial.
This ideal, however, seems ``too small'': For certain finite base fields, the group $I/I^2$ (with $I$ the fundamental ideal)
 is not even finitely generated \cite[10.5]{R}.

\section{$H$-forms}\label{sec:H-forms}

 Let $H$ be a subgroup of $k^\times/k^{\times d}$. We extend the definition of
$H$-forms given in \cite{H-P} to include the case that $H$ is an infinite subgroup:

\begin{definition} Let $H$ be a (not necessarily finite) subgroup of $ k^\times/k^{\times d}$.
A nondegenerate form $\varphi$ of degree $d$ over $k$
 is an {\it $H$-form}, if $H \subset G(\varphi)$; i.e., if $a\varphi\cong \varphi$
for all $a\in H$.
 A form $\Phi$ is called {\it $H$-reduced}, if $\Phi\cong \Psi\perp \varphi$ with an $H$-form $\varphi$
implies $\varphi=0$.
\end{definition}

\begin{remark} (i) If $H=\{1\}$ is trivial, then each form over $k$ is an $H$-form.\\
(ii)  Let $\varphi_1$ and $\varphi_2$ be two $H$-forms. Then the form
$\varphi_1\perp\varphi_2$ and the form $\Phi\otimes\varphi_1$ for any form $\Phi$ are also $H$-forms.
\\
(iii) Let $H$ be maximal, i.e. $H=H_{\rm max}= k^\times/k^{\times d}$.
  Then $\varphi$ is an $H$-form if and only if
 $G(\varphi)=D(\varphi)= k^\times/k^{\times d}$. Thus $\varphi$ is an $H$-form if and only if it is round and universal.
 For instance, the determinant of the $d$-by-$d$ matrices over $k$ is round and universal, and  the norm
 of a reduced Albert algebra is round and universal. The determinant of the $d$-by-$d$ matrices over $k$ is a form of degree $d$ and dimension $d^2$.  The norm
 of a reduced Albert algebra is a cubic form of dimension 27.
\\ (iv) Each hyperbolic quadratic form is an $H_{\rm max}$-form. However,
there  also exist anisotropic quadratic forms which are round and universal, so
 the quadratic $H_{\rm max}$-forms comprise a larger class than the quadratic hyperbolic forms.
\end{remark}

 \begin{example} Let $d=3$. Here are some examples of absolutely indecomposable cubic $H$-forms:\\
(i) If $l/k$ is a separable cubic field extension then its norm $\varphi(x)=n_{l/k}(x)$ is absolutely indecomposable
 and an $H$-form for $H=D(\varphi)$, the group of nonzero elements in $k$ that are represented by $\varphi$.
\\ (ii) If $A$ is a central simple $k$-algebra of degree $3$ then its reduced norm $\varphi(x)=n_{A/k}(x)$ is
 an $H$-form for $H=D(\varphi)$. If $A$ is split,
$\varphi$ is round, strongly universal and isotropic and an $H$-form for any $H$.
\\ (iii) Let $C$ be a unital composition algebra over $k$ with norm $n_C:C\rightarrow k$. If $A=k\times C$, and $\varphi(a+x)=a n_C(x)$ for all $a\in k,
x\in C$, then the cubic form $\varphi:A\rightarrow k$ is round, strongly universal, isotropic and an $H$-form for any choice of $H$.
\\ (iv) Let $A$ be a separable simple cubic Jordan algebra and $\varphi(x)=n_{A}(x)$ its norm; i.e., $A$ is
an Albert algebra of dimension 27 over $k$, or $A=H(B,\tau)$ with $B$ a central simple $l$-algebra of degree 3
with involution $\tau$, where $l=k$ and $\tau=id$, or where $l$ is a quadratic \'etale $k$-algebra with
conjugation $\tau$.
 The norm form of a reduced Albert algebra is trivially strongly universal, hence an $H$-form for any $H$.
(The norm of an Albert division algebra will in general not be universal.)
\\ (v) Let $(V,q)$ be a nondegenerate quadratic form over $k$. The cubic form $\varphi(a+v)=a q(v)$ defined for all $a
\in k, v\in V$ is strongly universal and isotropic, thus an $H$-form for any $H$.
\end{example}

\begin{theorem} \label{thm:1}
Let $d\geq 3$. Let $H\not=\{1\}$ be non-trivial. Each $H$-form of degree $d$ over $k$ is the orthogonal
sum of forms of the following kind:\\
(i) an indecomposable $H$-form;\\
(ii) an $H$-form of the type $\langle a_1,\dots,a_{s}\rangle \otimes \Psi$
($s\geq 2$,  $a_i\in H$), where $\Psi$ is an indecomposable form which is not an $H$-form, $s$ divides $|H|$,
and where $a_i\Psi\not\cong a_j\Psi$ for all $i\not=j$. One of the $a_i$
in $\langle a_1,\dots,a_{s}\rangle \otimes \Psi$  must be $1$.

\end{theorem}

\begin{proof} Let $\varphi$ be an $H$-form over $k$ of degree $d$ with Krull-Schmidt decomposition
$\varphi=\varphi_1\perp\dots\perp \varphi_r.$
 Let $r>1$ or else we are done. Suppose that $\varphi_1$ is not an $H$-form.
 As in \cite{H-P}, a form of the kind $a\varphi_1$, $a\in H$, is called a {\it conjugate}
of $\varphi_1$. Since $a\varphi\cong\varphi$ for all $a\in H$, we know that $a\varphi_1\in \{\varphi_1,\dots,\varphi_r\}$
for all $a\in H$ and so $\varphi_1$ has finitely many nonisomorphic conjugates $a_1\varphi_1,\dots, a_s\varphi_1$.
Here, $s\geq 2$: Since $\varphi_1$ is not an $H$-form, there exists $\widetilde{a}\in H$ such that $\widetilde{a}\varphi_1
\not\cong \varphi_1$ and $\widetilde{a}\varphi_1\cong \varphi_l$ for a suitable $l\in\{2,\dots,r\}$.
 Let $K=\{h\in H\,|\, h\varphi_1\cong\varphi_1\}$ be the stabilizer of $\varphi_1$, thus
 $H=a_1 K\cup \dots\cup a_sK$ is the disjoint union of the cosets. In particular, $s$ divides $|H|$, since
 $|H|=[H:K]|K|$.
Let $\Gamma=\perp a_i \varphi_1$ be the orthogonal sum of all the nonisomorphic conjugates of $\varphi_1$.
  By construction, $\Gamma$ is an $H$-form.
  Clearly, we have $\varphi\cong \Gamma\perp \omega$ for a suitable form
 $\omega$ over $k$. If $\omega=0$, then $\varphi=\Gamma=\langle a_1,\dots,a_{s}\rangle\otimes\varphi_1$
 with $\varphi_1$ an indecomposable form, which is not an
 $H$-form and $s> 1$. If $\omega\not=0$, then it is straightforward again that $\omega$ is an $H$-form, since both $\varphi$ and $\Gamma$
 are, and the assertion follows by induction.
 
   Since $1\in H$, one of the conjugates of $\varphi_1$
 is the form $\varphi_1$ itself. Hence, in (ii),  one of the $a_i$
in $\langle a_1,\dots,a_{s}\rangle \otimes \Psi$  must be $1$.
\end{proof}

\begin{proposition} \label{prop:1}
Let $d\geq 3$ and  $b\in k^\times$. \\
(i) The form $\langle b\rangle$ of degree $d$ over $k$ is an $H$-form if and only if $H=\{1\}$.\\
(ii) Let $H=\{a_1,\dots,a_s\}$ be non-trivial. Then
$$\langle a_1,\dots,a_s \rangle \otimes \langle b\rangle$$
is an $H$-form and there is no diagonal form $\varphi$ with entries in $H$ of dimension less than $s$ such that
$\varphi\otimes \langle b\rangle$ is an $H$-form. If $H=k^\times/k^{\times d}$ is maximal, then
$$\langle a_1,\dots,a_s \rangle \otimes \langle b\rangle\cong \langle a_1,\dots,a_s \rangle $$
is round and universal.\\
(iii) Let $H$ be non-trivial. If $\varphi$ is a diagonal $H$-form of degree $d$ over $k$, then its dimension is unequal to
$ld+1$ or $ld-1$ for an integer $l$ and $H$ must be finite.\\
 Let $H=\{a_1,\dots,a_s\}$. Then $\varphi$ is the orthogonal sum of $H$-forms of the kind
$$\langle a_1,\dots,a_s \rangle \otimes \langle b\rangle.$$
 In particular, if $H=k^\times/k^{\times d}$ is maximal, then
$\varphi$ is the orthogonal sum of round universal forms of the kind
$\langle a_1,\dots,a_s \rangle.$
\end{proposition}

\begin{proof} (i) $\langle b\rangle$  is an $H$-form if and only if
$h\langle b\rangle\cong\langle b\rangle$ for all $h\in H$. This is equivalent to $hb \equiv b\,{\rm mod}\, k^{\times d}$
for all $h\in H$; i.e.,
to $h\equiv 1\,{\rm mod}\, k^{\times d}$ for all $h\in H$.\\
(ii) is trivial since $a\langle b\rangle\not \cong\langle b\rangle$ for all $a\in H$, $a\not=1$.
Since multiplication by $b$
only permutes the entries of the form $\langle a_1,\dots,a_s \rangle$ for maximal $H$,
$\langle a_1,\dots,a_s \rangle \otimes \langle b\rangle\cong \langle a_1,\dots,a_s \rangle$.
\\
(iii) The assertion follows directly from \cite[Proposition 1]{Pu1}, Theorem 1 and (i): Let $\varphi=\langle b,\dots \rangle$. Since
$\langle b\rangle$  is not an $H$-form, there exists an element $a\in H$ such that $a\langle b\rangle\not\cong
\langle b\rangle$. We know that there can only be finitely many non-isomorphic conjugates of
$\langle b\rangle$. If $H$ were infinite, there would be infinitely many, because
$a\langle b \rangle\not\cong \tilde a \langle b \rangle$ for $a, \tilde a\in H$, $a\not =\tilde a$.
Thus $H$ is finite and we obtain the assertion.
\end{proof}

\begin{corollary} \label{cor:1}
 (i) For  infinite $H$, there are no diagonal $H$-forms of degree $d\geq 3$ over $k$.\\
(ii) Let $H$ be non-trivial and finite. If $\varphi$ is a diagonal cubic $H$-form then its dimension is a multiple of 3.
 If $\varphi$ is a diagonal quartic $H$-form then its dimension is a multiple of 2.\\
 (iii) Let $H=\{a_1,\dots,a_s\}$ be non-trivial. If $H$ is maximal, then
 $\langle a_1,\dots,a_s \rangle \otimes \Psi$
 is an $H$-form, for every form $\Psi$ over $k$.
\end{corollary}

\begin{proof} It remains to show (iii), since (i) and (ii) are obvious:\\
By Proposition \ref{prop:1}, $\langle a_1,\dots,a_s \rangle $ is an $H$-form and thus round and universal.
Since the tensor product of a round universal form with any other
one again is round and universal, the form $\langle a_1,\dots,a_s \rangle \otimes \Psi$ is an $H$-form.
\end{proof}

\begin{proposition} \label{prop:2}
Let $d\geq 3$.\\
(i) An indecomposable separable form $tr_{l/k}\langle c\rangle$ with $c\in l^\times$ is an $H$-form if and only if
$l$ is a separable field extension of $k$ such that $h\equiv 1 \,{\rm mod}\, l^{\times d}$ for all elements $h\in H$.\\
(ii)  Let $H=\{a_1,\dots,a_s\}$ be non-trivial. Each separable $H$-form of degree $d$ over $k$
is the orthogonal sum of forms of the following kind:\\
(a) $\langle a_1,\dots,a_{{s}}\rangle \otimes \langle b\rangle,$ where $ b\in k^\times$.\\
(b) An $H$-form of the kind $tr_{l/k}\langle c\rangle$ with $c\in l^\times$.\\
(c) $\langle a_{1},\dots,a_{s}\rangle \otimes tr_{l/k}\langle c\rangle,$
where $tr_{l/k}\langle c\rangle$ is an indecomposable form which is not an $H$-form
(i.e., $l$ is a separable field extension of $k$ such that $h\equiv 1 \,{\rm mod}\, l^{\times d}$
 for all elements $h\in H$).
\end{proposition}

\begin{proof} (i) An indecomposable form of the kind $tr_{l/k}\langle c\rangle$ with $c\in l^\times$ is an $H$-form if and only if
$h tr_{l/k}\langle c\rangle\cong tr_{l/k}\langle hc\rangle\cong tr_{l/k}\langle c\rangle$ for all $h\in H$, which is
equivalent to $\langle hc\rangle\cong \langle c\rangle$ \cite[3.12]{H-P}. This in turn is the case if and only if we have
$h\equiv 1 \,{\rm mod}\, l^{\times d}$ for all elements $h\in H$.\\
(ii) The assertion follows again as a direct consequence of Theorem \ref{thm:1} and (i). Using the Theorem of Krull-Schmidt, it remains to
study the case that $\varphi\cong tr_{l/k}(\langle c\rangle)\perp\dots$, where $tr_{l/k}(\langle c\rangle)\perp\dots $ is indecomposable
and not an $H$-form. Then $\langle a_{m_1},\dots,a_{m_l}\rangle \otimes tr_{l/k}\langle c\rangle$ is a subform of $\varphi$,
where $m_l> 1$ divides $|H|$, $a_{m_i}\not \equiv a_{m_j}{\rm mod}\, l^{\times d}$ for $i\not=j$, and $a_{m_i}\in H$
for all $i$. Since no non-trivial element of $H$ becomes a $d$th power in $l^\times$  by (i), however,
the forms are of the kind
$\langle a_{1},\dots,a_{s}\rangle \otimes tr_{l/k}\langle c\rangle,$
where $tr_{l/k}\langle c\rangle$ is an indecomposable form which is not an $H$-form.
\end{proof}

\begin{example} (i) Let $k$ be a field such that $k^\times/k^{\times d}$ has prime order $p>1$.
 Put $H=k^\times/k^{\times d}=\{1,a,\dots,a^{p-1}\}$.
Let $\varphi$ be a form over $k$. If $G(\varphi)=k^\times/k^{\times d}$,
then $\varphi$ is  round and universal and thus an $H$-form.

 If $\varphi$ is a diagonal $H$-form of degree $d$ over $k$, then
$\varphi\cong m\times \langle 1,a,\dots,a^{p-1}\rangle$.

If $\varphi$ is a separable $H$-form of degree $d$ over $k$, then it is the orthogonal sum of forms of the kind
$ \langle 1,a,\dots,a^{p-1}\rangle$, forms of the kind $\langle 1,a,\dots,a^{p-1}\rangle \otimes tr_{l/k}\langle c\rangle,$
with $c\in l^\times$ arbitrary ($l$ a field extension of $k$ such that $h\not\equiv 1 \,{\rm mod}\, l^{\times d}$
for all elements $h\in H$),
and forms of the kind
$tr_{l'/k}\langle b\rangle$ with $b\in l'^\times$ (and any $b\in l'^\times$ works here),
 provided that
$l'$ is a field extension of $k$ such that $h\equiv 1 \,{\rm mod}\, l'^{\times d}$ for all elements $h\in H$.
\\ (ii)  Let $k=\mathbb{F}_q$ be a finite field ($q=p^t$) and $d=r$ a prime number such that $p>r$.
If $q\equiv 1\,{\rm mod}\,r$ (i.e., if $k$ contains a primitive $r^{th}$ root of unity $\xi$),
then $k^\times/k^{\times r}$ is cyclic of order $r$, otherwise it is trivial.
Suppose that $q\equiv 1\,{\rm mod}\, r$
and choose $H=k^\times/k^{\times r}=\{1,a,\dots,a^{r-1}\}$ (if, additionally, $r^2$ does not divide $q-1$,
 even $k^\times/k^{\times r}=\{1,\xi,\dots,\xi^{r-1}\}$).
Hence $H$-forms over $k$  are exactly the universal round forms over $k$.

$\varphi$ is a diagonal universal round form  of degree $r$ if and only if
$$\varphi\cong m\times \langle 1,a,\dots,a^{r-1}\rangle=
\langle 1,a,\dots,a^{r-1}\rangle\perp\dots\perp\langle 1,a,\dots,a^{r-1}\rangle$$
for some integer $m$ and ${\rm dim}\,\varphi$ is a multiple of the degree. Moreover, if $-1\in {\mathbb F}_q^{\times r}$ and $r \geq 4$, then $u_{\rm diag}(r,k)\leq r-1$ \cite{O}, each $H$-form is isotropic.

 An indecomposable separable form $ tr_{l/k}\langle c\rangle$ is universal and round if and only
if $h\equiv 1 \,{\rm mod}\, l^{\times r}$ for all elements $h\in k^\times$. However, this cannot happen for any choice of
$c\in l^\times$, because $|l^\times/l^{\times r}|=r$. Therefore no such form is an $H$-form.

Any separable $H$-form of degree $r$ over $k$  is the orthogonal sum of forms of the kind
\begin{enumerate}
\item $\langle 1,a,\dots,a^{r-1}\rangle$ (which all are round and universal);
\item $\langle 1,a,\dots,a^{r-1}\rangle \otimes tr_{l/k}\langle c\rangle,$
with $c\in l^\times$ arbitrary
\end{enumerate}
\end{example}

\section{The Witt ring}\label{sec:Witt ring}

Let $d\geq 2$ and let $H$ be a subgroup of $k^\times/k^{\times d}$.
For forms of higher degree the following construction was introduced by Harrison-Pareigis \cite{H-P}  for  finite $H$.
 We also allow infinite $H$ (the proofs remain the same for infinite $H$):

For each form $\Phi$ there is an $H$-reduced form $\Phi_H$ and an $H$-form $t_H(\Phi)$ such that
$\Phi\cong \Phi_H \perp t_H(\Phi)$. For $d>2$ this decomposition is unique up to isomorphisms \cite[1.1, 1.2]{H-P}.
 Two forms $\Phi$ and $\Psi$ are {\it $H$-equivalent}, written $\Phi \sim_H \Psi$, if there are $H$-forms
$\varphi_1$ and $\varphi_2$ such that $\Phi\perp \varphi_1\cong \Psi\perp \varphi_2$.
The equivalence relation $\sim_H$ is compatible with taking orthogonal sums and tensor products, thus
orthogonal sums and the tensor product induce an addition and a multiplication on the
 set of equivalence classes
$$W_d(k,H)=\{[\Phi]\,|\,\Phi \, \text{a form over}\, k\}.$$
For $d>1$, each equivalence class $[\Phi]$
 in $W_d(k,H)$ contains at least one $H$-reduced representative $[\Phi]=[\Phi_H]$.
 For each $d>2$, the $H$-reduced representative in each equivalence class is uniquely determined, so $W_d(k,H)$
 can be viewed as subset of $P_d(k)$ \cite[1.3, 1.4]{H-P}. The canonical map
$$r_H:P_d(k)\rightarrow  W_d(k,H), \quad \Phi \mapsto [\Phi]$$
 is surjective and preserves addition and multiplication. Hence $W_d(k,H)$ inherits the
structure of $P_d(k)$ and is a commutative semi-ring with addition given by orthogonal sums and
 multiplication by the tensor product. For finite $H$, $W_d(k,H)$ is a commutative ring
(the finiteness of $H$ is needed for proving the existence of
an additive inverse for any form $\Phi$ over $k$).
If $H$ is infinite, not every element has an additive inverse:

\begin{proposition}\label{prop:3}
Let $H$ be an infinite subgroup of $k^\times/k^{\times d}$.\\
(i) Let $\Phi$ be a form over $k$ such that there are only finitely many
elements $a_1,\dots, a_r\in k^\times/k^{\times d}$ such that $a_i \Phi \not \cong \Phi$. Then
$\Phi$ has the additive inverse $\widetilde{\Phi}=\bot_{i=1}^r a_i\Phi$ in $W_d(k,H)$.
\\(ii) Diagonal forms have no additive inverse in $W_d(k,H)$, if $d\geq 3$.
\end{proposition}

\begin{proof} (i) The group $H$ canonically acts on $P_d(k)$. Let $\Phi$ be a form over $k$
such that there are only finitely many
elements $a_i\in k^\times/k^{\times d}$ with $a_i \Phi \not \cong \Phi$. Then $K=G(\Phi)=
\{h\in H \,|\, h\Phi\cong \Phi \}$ is the stabilizer of $\Phi$ and $H=a_1 K\cup \dots \cup a_n K$ is
 the disjoint (by assumption finite) union of its cosets. $H$ acts (on the left) on the set of (left) cosets of $K$.
Thus $\Phi\perp \widetilde{\Phi}=\Phi\perp (\perp_{a_i}\Phi)$ is an $H$-form over $k$.\\
(ii) The additive inverse of a diagonal form again is a diagonal form.
 For infinite $H$ there are no diagonal $H$-forms of degree $d\geq 3$ over $k$ (Corollary \ref{cor:1}).
\end{proof}

\begin{remark} (i)  Only forms $\varphi$ with finite index $[H:G(\varphi)]$ have an additive inverse.\\
(ii) At first glance, the case that the subgroup $H$ is largest possible looks like  the most
interesting one, since in this case $H$-forms are exactly the round universal forms, which resembles the situation
in the classical case. We will see later that this does not seem to be the case (see Propositions \ref{prop:3} (iii), \ref{prop:4}(iii)).
 However, we immediately note that an $H_{\rm max}$-form does not need to stay an $H_{\rm max}$-form
under field extensions. Thus we cannot define a base change homomorphism $W_d(k,H_{\rm max})\rightarrow  W_d(l,H_{\rm max})$.

It is possible, however, to define traces. For $d> 2$ the form $s_*(\varphi)$ is an
$H_{\rm max}$-form for each $H_{\rm max}$-form $\varphi$ over $l$. Thus $s_*$ induces a homomorphism of the semigroups
$$W_d(l,H_{\rm max})\rightarrow  W_d(k,H_{\rm max})$$
 for any $k$-linear map $s:l\rightarrow  k$.\\
(iii) Another way to describe $W_d(k,H_{\rm max})$ is the following: Let $F$ be the semi-ideal in the Witt-Grothendiek ring
$ \widehat{W}_d(k)$ generated by all round universal forms over $k$, then $W_d(k,H_{\rm max})= \widehat{W}_d(k)/F$.
If $H_{\rm max}$ is finite, $F$ is a subgroup, not just a semisubgroup, hence an ideal.
\end{remark}

 There seem to be some disadvantages to the definition of $W_d(k,H)$. Initially, the most striking perhaps is the fact that
in the special case that the degree is 2; i.e., in case we look at quadratic forms, the classical Witt ring $W(k)$
of quadratic forms over $k$ in general does not coincide with any Witt ring of the kind $W_d(k,H)$.
 More precisely, there is a canonical surjective
ring homomorphism $W(k)\rightarrow  W_d(k,H)$ which has the annihilator
${\rm Ann}_{W(k)}\{\langle 1\rangle- \langle h\rangle \,|\, h\in H\}$ as its kernel \cite{H-P}.
This is due to the fact that certain anisotropic quadratic forms are also quadratic $H$-forms,
not only hyperbolic quadratic forms.

When we consider equivalent diagonal forms, it suffices to only use diagonal $H$-forms:

\begin{theorem} \label{thm:2}
Let $H=\{a_1,\dots,a_s\}$ be non-trivial. \\
(i) $\langle b_1,\dots,b_m\rangle\sim_H \langle d_1,\dots,d_s\rangle$
if and only if there are two diagonal $H$-forms $\Phi$ and $\Psi$ such that
$\langle b_1,\dots,b_m\rangle\perp  \Phi \cong \langle d_1,\dots,d_s\rangle\perp  \Psi.$
\\ (ii) $\langle b_1,\dots,b_m\rangle\sim_H \langle d_1,\dots,d_s\rangle$
if and only if
$$\begin{array}{l}
\langle b_1,\dots,b_m\rangle \perp  \langle e_1\rangle \otimes\langle a_{1},\dots,a_{s}\rangle
\perp\dots\perp  \langle e_n\rangle\otimes\langle a_{1},\dots,a_{s}\rangle \cong \\
\langle d_1,\dots,d_s\rangle\perp
 \langle c_1\rangle\otimes\langle a_{1},\dots,a_s\rangle
\perp\dots\perp  \langle c_t\rangle \otimes\langle a_{1},\dots,a_{s}\rangle
\end{array}$$
for suitable $e_i,c_i\in k^\times$.
\\ (iii) Let $H$ be maximal. Then
$\langle b_1,\dots,b_m\rangle\sim_H \langle d_1,\dots,d_s\rangle$
if and only if
$$\langle b_1,\dots,b_m\rangle\cong \langle d_1,\dots,d_s\rangle\perp l\times \langle a_1,\dots,a_s \rangle$$
or
$$\langle d_1,\dots,d_s\rangle\cong \langle b_1,\dots,b_m\rangle\perp l\times \langle a_1,\dots,a_s \rangle$$
for some $l\geq 0$.
\end{theorem}

\begin{proof} The main arguments used in the proof are the Theorem of Krull-Schmidt, Theorem \ref{thm:1} and Corollary \ref{cor:1}.
Let $\langle b_1,\dots,b_m\rangle\sim_H \langle d_1,\dots,d_s\rangle$.
By definition and by Theorem \ref{thm:1},
this means there are $H$-forms
$$\Phi\cong \widetilde{\Phi_1}\perp \dots\perp  \widetilde{\Phi_r}\perp \Phi_1\otimes \langle a_{i_1},\dots,a_{i_t}\rangle
\perp\dots\perp \Phi_v\otimes \langle a_{j_1},\dots,a_{j_w}\rangle$$
and
$$\Psi\cong \widetilde{\Psi_1}\perp \dots\perp  \widetilde{\Psi_l}\perp \Psi_1\otimes\langle a_{i_1},\dots,a_{i_y}\rangle
\perp\dots\perp \Psi_u \otimes\langle a_{j_1},\dots,a_{j_z}\rangle$$
with the $\widetilde{\Phi_i}$'s and the $\widetilde{\Psi_i}$'s being indecomposable $H$-forms, hence of dimension greater than
1, and with the $\Phi_i$'s and the $\Psi_i$'s being indecomposable forms which are no $H$-forms,
such that $\langle b_1,\dots,b_m\rangle \perp \Phi\cong \langle d_1,\dots,d_s\rangle\perp \Psi$. By Krull-Schmidt,
this implies that $r=l$ and
$\widetilde{\Phi_1}\perp \dots\perp  \widetilde{\Phi_r}\cong \widetilde{\Psi_1}\perp \dots\perp  \widetilde{\Psi_l}$
(all are indecomposable forms of dimension greater than 1 and $H$-forms), thus
$$\begin{array}{l}
\langle b_1,\dots,b_m\rangle \perp  \Phi_1\otimes\langle a_{i_1},\dots,a_{i_t}\rangle
\perp\dots\perp \Phi_v\otimes\langle a_{j_1},\dots,a_{j_w}\rangle \cong \\
 \langle d_1,\dots,d_s\rangle\perp
\Psi_1\otimes\langle a_{i_1},\dots,a_{i_y}\rangle
\perp\dots\perp \Psi_u \otimes\langle a_{j_1},\dots,a_{j_z}\rangle.
\end{array}$$
By Krull-Schmidt again, we can ``sort out'' also all those indecomposable
components containing $\Phi_i$'s and $\Psi_i$'s which have dimension greater than 1 now (their sums must be isomorphic as well) and
obtain that
$$\begin{array}{l}
\langle b_1,\dots,b_m\rangle \perp  \Phi_{i_1}\otimes\langle a_{1},\dots,a_{s}\rangle
\perp\dots\perp \Phi_{i_m}\otimes\langle a_{1},\dots,a_{s}\rangle \cong \\
\langle d_1,\dots,d_s\rangle\perp
\Psi_{h_1}\otimes\langle a_{1},\dots,a_s\rangle
\perp\dots\perp \Psi_{h_t} \otimes\langle a_{1},\dots,a_{s}\rangle
\end{array}$$
with all the remaining $\Phi$'s and $\Psi$'s of dimension 1, by Proposition \ref{prop:1} (ii). This proves (i) and (ii).
 Applying Krull-Schmidt
again and using Proposition \ref{prop:1}(iii), (iii) is  trivial.
\end{proof}

\begin{proposition} \label{prop:4}
Let $H=\{a_1,\dots,a_s\}$ be non-trivial. Let $\varphi_1$ and $\varphi_2$ be two separable forms,
$\varphi_1\cong \langle b_1,\dots,b_m\rangle\perp \varphi'_1$ and $\varphi_1\cong \langle d_1,\dots,d_s\rangle\perp
\varphi'_2$ where
$m\geq 0$, $s\geq 0$, and where the remaining part $\varphi'_i$ of $\varphi_i$ consists of indecomposable components of dimension greater than 1. \\
(i) $\varphi_1\sim_H \varphi_2$ if and only if there are two separable $H$-forms $\Phi$ and $\Psi$ such that
$\varphi_1\perp  \Phi \cong \varphi_2\perp  \Psi.$
\\ (ii) $\varphi_1\sim_H \varphi_2$ if and only if  $\langle b_1,\dots,b_m\rangle\sim_H \langle d_1,\dots,d_s\rangle$
and there are two separable $H$-forms $\Phi'$ and $\Psi'$, which do not contain any
indecomposable components of dimension 1, such that
$$\varphi'_1\perp  \Phi' \cong \varphi'_2\perp  \Psi'.$$
\\ (iii) Let $H$ be maximal. Then
 $\langle b_1,\dots,b_m\rangle\sim_H \langle d_1,\dots,d_s\rangle$
if and only if
$$\langle b_1,\dots,b_m\rangle\cong \langle d_1,\dots,d_s\rangle\perp l\times \langle a_1,\dots,a_s \rangle$$
or
$$\langle d_1,\dots,d_s\rangle\cong \langle b_1,\dots,b_m\rangle\perp l\times \langle a_1,\dots,a_s \rangle$$
for some $l\geq 0$ and
$$\varphi'_1 \cong \varphi'_2\perp  \Psi' \text{  or }\varphi'_2 \cong \varphi'_1\perp  \Phi'$$
for a suitable separable $H$-form $\Phi'$ or $\Psi'$ which does not contain any indecomposable component of dimension 1.
\end{proposition}

\begin{proof} The main arguments in the proof are again the Theorem of Krull-Schmidt and Theorems  \ref{thm:1} and Proposition  \ref{prop:1}.
The proof is then analogous to the one of Theorem \ref{thm:2}. 
By definition and by Theorem \ref{thm:1}, $\varphi_1\sim_H \varphi_2$ means there are $H$-forms
$$\Phi\cong \widetilde{\Phi_1}\perp \dots\perp  \widetilde{\Phi_r}\perp \Phi_1\otimes \langle a_{i_1},\dots,a_{i_t}\rangle
\perp\dots\perp \Phi_v\otimes \langle a_{j_1},\dots,a_{j_w}\rangle$$
and
$$\Psi\cong \widetilde{\Psi_1}\perp \dots\perp  \widetilde{\Psi_l}\perp \Psi_1\otimes\langle a_{i_1},\dots,a_{i_y}\rangle
\perp\dots\perp \Psi_u \otimes\langle a_{j_1},\dots,a_{j_z}\rangle$$
with the $\widetilde{\Phi_i}$'s and the $\widetilde{\Psi_i}$'s being indecomposable $H$-forms, hence of dimension greater than
1, and with the $\Phi_i$'s and the $\Psi_i$'s being indecomposable forms which are no $H$-forms,
such that $\varphi_1 \perp \Phi\cong \varphi_2\perp \Psi$.
By Krull-Schmidt,
we can first eliminate all indecomposable components which are not separable (their respective sums must be isomorphic),
proving (i).
By Krull-Schmidt again, we can ``sort out'' also all those indecomposable
components of $\Phi$ and $\Psi$ which have dimension greater than 1 now (their sums, including $\varphi_1$ and $\varphi_2$,
 must be isomorphic as well) and thus, with  Theorem \ref{thm:2} 
 obtain (ii) and (iii).
\end{proof}

The Theorem of Krull-Schmidt  does not allow for any interaction between indecomposable components
of different dimensions. Thus it makes sense to only investigate Witt rings of $H$-equivalent diagonal forms or separable forms, etc., and hope to get results
on the structure theory this way, and to also restrict the equivalence notions to the $H$-forms of corresponding type.

Similar observations hold for forms which are the sum of indecomposable components of
other types, e.g. of forms which are the
sums of traces of indecomposable forms of dimension 2 (or higher), or of absolutely indecomposable forms:

\begin{proposition} \label{prop:5}
Let $H=\{a_1,\dots,a_s\}$ be non-trivial. Let $\varphi_1$ and $\varphi_2$ be two forms, whose indecomposable
components are absolutely indecomposable. Then
 $\varphi_1\sim_H \varphi_2$ if and only if there are two $H$-forms
$$\Phi\cong \widetilde{\Phi_1}\perp \dots\perp  \widetilde{\Phi_r}\perp \Phi_1\otimes \langle a_{i_1},\dots,a_{i_t}\rangle
\perp\dots\perp \Phi_v\otimes \langle a_{j_1},\dots,a_{j_w}\rangle$$
and
$$\Psi\cong \widetilde{\Psi_1}\perp \dots\perp  \widetilde{\Psi_l}\perp \Psi_1\otimes\langle a_{i_1},\dots,a_{i_y}\rangle
\perp\dots\perp \Psi_u \otimes\langle a_{j_1},\dots,a_{j_z}\rangle$$
with the $\widetilde{\Phi_i}$'s and the $\widetilde{\Psi_i}$'s being absolutely indecomposable $H$-forms, hence of dimension greater than
1, and with the $\Phi_i$'s and the $\Psi_i$'s being absolutely indecomposable forms which are no $H$-forms, such that
$\varphi_1\perp  \Phi \cong \varphi_2\perp  \Psi$.
\end{proposition}

\subsection{The case that $H=\{1,-1\}$}

 Let $d$ be even and $H=\{1,-1\}$. If $\Phi$ is a separable $H$-form, then it is an orthogonal sum of forms of the
kind
\begin{enumerate}
\item $\langle a,-a\rangle$,
\item $tr_{l/k}\langle b,-b\rangle$ with $l/k$ a separable field extension such that $|l^\times/l^{\times d}|\not=1$,
\item $tr_{l/k}\langle c\rangle$ with $l/k$ a separable field extension such that $|l^\times/l^{\times d}|=1$.
\end{enumerate}
By Theorem \ref{thm:2} 
$\langle b_1,\dots,b_m\rangle\sim_H \langle d_1,\dots,d_s\rangle$
if and only if
$$\begin{array}{l}
\langle b_1,\dots,b_m\rangle \perp  \langle e_1\rangle \otimes\langle 1,-1\rangle
\perp\dots\perp  \langle e_n\rangle\otimes\langle 1,-1\rangle \cong \\
\langle d_1,\dots,d_s\rangle\perp
 \langle c_1\rangle\otimes\langle 1,-1\rangle
\perp\dots\perp  \langle c_t\rangle \otimes\langle 1,-1\rangle
\end{array}$$
for suitable $e_i,c_i\in k^\times$. Hence, if $\langle b_1,\dots,b_m\rangle$ is anisotropic, it is $H$-reduced and
$\langle b_1,\dots,b_m\rangle\sim_H \langle d_1,\dots,d_s\rangle$ implies that $\langle b_1,\dots,b_m\rangle\cong
\langle d_1,\dots,d_s\rangle$ or $\langle d_1,\dots,d_s\rangle\cong \langle b_1,\dots,b_m\rangle \perp  \langle
e_{i_1}\rangle \otimes\langle 1,-1\rangle \perp\dots\perp  \langle e_{i_k}\rangle\otimes\langle 1,-1\rangle$.

\begin{remark} Suppose $n\times\langle 1\rangle\sim_H 0$. Then $n\times\langle 1\rangle\cong \langle e_1,-e_1,\dots,e_m,-e_m
\rangle$ implies the contradiction that $1\equiv -1\, {\rm mod}\, k^{\times d}$. Hence  for any
field $k$, $W_d^D(k,H)$ is not a torsion group. For $d=2$, this is only true if $k$ is formally real.
\end{remark}

\section{Some observation on a Witt semiring of diagonal forms and invariants of diagonal forms}\label{sec:4}

  If we want to define a Witt ring for diagonal forms, we need to specify which of them are going to be the ``hyperbolic'' ones, i.e. the forms we will cancel out in our definition. For this we have to take into consideration that there are no strongly
multiplicative forms of degree $d\geq 3$ which are diagonal other than
$\varphi\cong\langle 1\rangle$ \cite{Pu2}, and that there are no diagonal forms which are round and universal
and remain so under each field extension. So
it makes sense to define a ``hyperbolic'' diagonal form as a form which is round and universal.

 We also have to assume that
 $k^\times/k^{\times d}$ is finite, or else
there are no diagonal forms which are round and universal. This again takes us back to Harrison's and Pareigis'
original definition of the Witt ring, where $H$ is chosen to be maximal.

Let now $P^D_d(k)$ denote the set of all isomorphism classes of nondegenerate  diagonal forms of degree $d$ over $k$.
Define $$W^D_d(k,H)=P^D_d(k)/\sim_H$$
and let $$r^D_H:P^D_d(k)\rightarrow  W^D_d(k,H), \quad \Phi \rightarrow  [\Phi]$$
 denote the canonical map.
 Then $W^D_d(k,H)$ is a commutative semi-ring. $W^D_d(k,H)$ is the sub semiring of diagonal forms of $W_d(k,H)$; i.e.,
 $$W^D_d(k,H)=\{[\Phi]\in W_d(k,H)\,|\, \Phi \text{ a diagonal form } \}.$$
A diagonal form $\Psi$ of degree $d$ over $k$ has an additive inverse (which is again a diagonal form)
  if and only if there are only finitely many
  elements $a_i\in H$ such that $a_i\Psi\not=\Psi$; i.e., if its stabilizer has finite index in $H$. Thus  if $H$ is finite then $W^D_d(k,H)$
  is a ring.

 \begin{remark} An $H$-form can be isotropic or anisotropic:
   Let $d$ be an odd integer which is relatively prime to both $p$ and $q-1$,
where $q=p^r$ for some $r$. Then $u_{diag}(d,\mathbb{Q}_p)=|\mathbb{Q}_p^{\times}/\mathbb{Q}_p^{\times d}|=d$.
 The minimal (and maximal) dimension of a diagonal
anisotropic universal form of degree $d$ over $\mathbb{Q}_p$ is $d$. This form is uniquely determined (up to
isomorphism), round,  and given by $\langle  a_1,...,a_d\rangle$ \cite{Pu1}. For $m\geq 2$ the $H$-form
$m\times \langle a_1,\dots,a_d\rangle$ is isotropic.
\end{remark}

From now on and unless explicitly stated otherwise,  choose $$H=H_{max}=k^\times/k^{\times d}$$
to be maximal and suppose $H$ is non-trivial and finite. Let $\{a_1,\dots, a_s\}$ be a set of coset representatives of $k^\times/k^{\times d}$. 
Every diagonal $H$-form
$\varphi$ of degree $d$ is isomorphic to a form of the kind
$m\times \langle a_1,\dots,a_s\rangle$
 ($m$ any integer) and $ms={\rm dim }\,\varphi\not \in\{rd-1,rd+1\}$.

In this paper, we will think of the forms we want to cancel out in our definitions of Witt rings as the ``hyperbolic forms''. That means, here these are the $H$-forms, and the dimension of each such ''hyperbolic'' $H$-form is a multiple of its degree $p$. 
Since the dimension of each $H$-form is a multiple of its degree $p$, this allows us to define $\widetilde W_d(k, H )=W^D_d(k,H)$
and the map
$${\rm dim}: \widetilde W_d(k, H)\rightarrow  \mathbb{Z}_d.$$
which is called the {\it dimension index}. Put
 $$I_d={\rm ker}({\rm dim}),$$
  then $I_d$ is an ideal in the ring $ \widetilde W_r(k, H)$. A form $\Phi$ lies in $I_d$ if and only if
its dimension is a multiple of $d$. Thus as a group $I_d$ is generated by forms of the kind
$\langle b_1,\dots,b_d \rangle$
 with $b_i\in k^\times$, $I_d^2$ is generated by forms of the kind
$$\langle b_{1,1},\dots,b_{d,1}\rangle \otimes \langle b_{1,2},\dots,b_{d,2}\rangle$$
 and $I_d^n$ by forms of the kind
$$\langle b_{1,1},\dots,b_{d,1}\rangle \otimes \dots\otimes\langle b_{1,1},\dots,b_{d,n} \rangle$$
 with $b_{i,j}\in k^\times$.

 Given the space $M_n(k)$ of $n\times n$ matrices over $k$, the {\it permanent}
 ${\rm per}: M_n(k)\rightarrow  k$ is defined by the modified Leibniz formula
 $${\rm per}(a_{ij})=\sum_{\sigma\in S_n}\prod_{i=1}^{n}a_{i \sigma(j)}.$$
The permanent $d={\rm per}$ is a first degree cohomological invariant for separable forms and replaces the
 discriminant in this setting. For two separable forms
 $ \Phi,\Psi$ we have $d(\Phi\perp\Psi)=d(\Phi)d(\Psi)$ and
 $d(\Phi\otimes\Psi)=d(\Phi)^{{\rm dim} \Psi} d(\Psi)^{{\rm dim} \Phi}$
 \cite[4.15]{R}.

Since each diagonal $H$-form also has permanent $1$, the permanent induces
a surjective homomorphism
$$d:\widetilde W_d(k,H) \rightarrow  k^\times/k^{\times d}.$$
We have
$d(\langle b_1,\dots,b_d\rangle \otimes \langle c_1,\dots,c_d \rangle)=1$,
and thus obtain that $d:I_d/I_d^2\rightarrow  k^{\times}/k^{\times d}$
is a surjective group homomorphism:

\begin{lemma}\label{le:1}
The map
$d:I_d\rightarrow  k^{\times}/k^{\times d}$
 is a surjective group homomorphism and $I_d^2\subset {\rm ker}\,d$.
\end{lemma}

Indeed, the above definition of $\widetilde W_d(k, H)=P^D_d(k)/\sim_{\widetilde H}$ as
well as  Lemma \ref{le:1} 
stay true if we take any finite subgroup $H$ of
$k^\times/k^{\times d}$ with the corresponding requirements. It does not have to be maximal.

For the remainder of this section, let $k$ be a field such that $k^\times/k^{\times p}=\{1,a,\dots,a^{p-1}\}$ for a suitable $a\in K^\times$. Then any diagonal $H$-form of prime degree $p$
is isomorphic to a form of the type $m\times \langle 1,a,\dots,a^{p-1}\rangle$ and
$$d:I_p/I_p^2\rightarrow  k^{\times}/k^{\times p}$$
is a surjective group homomorphism.

\begin{lemma}  Let $p>2$ be a prime and $k^\times/k^{\times p}=\{1,a,\dots,a^{p-1}\}$. 
Then the factor group
$I_p^n/I_p^{n+1}$ has exponent $p$.
\end{lemma}

\begin{proof} We have
$p\times (\langle a_1,\dots,a_p\rangle\otimes\dots\otimes\langle s_1,\dots,s_p\rangle)
\cong \langle a_1,\dots,a_p\rangle\otimes\dots\otimes\langle s_1,\dots,s_p\rangle\otimes\langle 1,\dots,1\rangle\in I_p^{n+1},$
hence is $0$ in $I_p^n/I_p^{n+1}$.
\end{proof}

 If $p=3$ we compute the following result for diagonal cubic forms over $k$:

\begin{theorem} \label{thm:3}
Let $k$ be a field such that $k^\times/k^{\times 3}=\{1,a,a^2\}$. 
Then the permanent $d$ induces a group isomorphism
$$I_3/I_3^2\cong k^{\times}/k^{\times 3}.$$
\end{theorem}

\begin{proof} It remains to show that each cubic form $\Phi$ over $k$ which lies in $I_3$ and has $d(\Phi)=1$
satisfies $\Phi\in I_3^2$.
We proceed by induction on the dimension on $\Phi$. If ${\rm dim}\,\Phi=3$, then $\Phi=\langle a,a,b\rangle$ (not all
entries can be different, since otherwise it would be an $H$-form) and $d(\langle a,a,b\rangle)=1$
if and only if $b\equiv a\,{\rm mod}\, k^{\times 3}$, so $\Phi=\langle a,a,a\rangle$.
 However,
$\langle a,a,a \rangle\in I_3^2$, because
$$\langle a,a,a\rangle\sim_H \langle 1,a,a\rangle \perp \langle 1,a,a\rangle
\perp \langle a,a^2,a^2\rangle\cong \langle 1,1,a\rangle\otimes\langle 1,a,a\rangle \in I_3^2.$$
 Let $\Phi=\langle a_1,a_1,b_1,c_1,c_1,d_1,\dots,g_1,g_1,h_1\rangle$ be a form of dimension $3m=n$. Then
$$\langle a_1,a_1,b_1,c_1,c_1,d_1,\dots,g_1,g_1,h_1\rangle\sim_H $$
$$\langle a_1,a_1,b_1\rangle\perp\langle 1,a_1,a_1^2\rangle
\perp\dots\perp \langle c_1,c_1,d_1,\dots,g_1,g_1,h_1\rangle\cong $$
$$\langle a_1,a_1,a_1 \rangle
\perp \langle 1,a_1^2,b_1 \rangle \perp\langle c_1,c_1,d_1,\dots,g_1,g_1,h_1\rangle.$$
 Put
 $$\varphi_2=\langle 1,a_1^2,b_1 \rangle \perp\langle c_1,c_1,d_1,\dots,g_1,g_1,h_1\rangle.$$
  Now $\varphi_1=\langle a_1,a_1,a_1 \rangle
\in I_3^2$, hence $d(\varphi_1)=1$ by the induction beginning, and we obtain
$d(\varphi_2)=d(\Phi)/d(\varphi_1)=1/1=1$. By induction hypothesis, $\varphi_2\in I_3^2$.
\end{proof}

The calculations needed to prove that the above result holds more generally
and not just for $r=3$ appear to be tedious.

In the following examples, the permanent $d$ induces a group isomorphism
$$I_3/I_3^2\cong k^{\times}/k^{\times 3}$$
 for diagonal cubic forms.

\begin{example} (i) Let $p$ be an odd prime, $\omega$ a primitive $p$th root of unity and $\rho$ a
primitive $p^{2}$-nd root of unity. Let $k$ be a subfield of the algebraic numbers $\mathbb{Q}^a$ containing $\omega$
which is maximal with respect to the property of not containing $\rho$. Then $|k^\times/k^{\times p}|=p$
and $\{1,\omega,\dots,\omega^{p-1}\}$ is a set of coset representatives for $k^\times/k^{\times p}$.
The $H$-form $\langle 1,\omega,\dots,\omega^{p-1}\rangle$ is isotropic and  $ u_{diag}(k,p)\leq p-1$
\cite[Example 3.2]{L-Y}.
\\(ii) Let $\mathbb{Q}^a$ be the algebraic numbers and $\omega$ a primitive $p$th root of unity, $p$ an odd prime.
Let $k$ be a subfield of $\mathbb{Q}^a$ which is maximal with respect to the property of not containing
$\omega$ and $\alpha$, where $\mathbb{Q}(\alpha)/\mathbb{Q}$ is a normal extension of degree $p$. Then
${\rm Gal}(k)=\mathbb{Z}_p\times \mathbb{Z}_q$ where $q$ does not divide $p-1$,
$|k^\times/k^{\times p}|=1$ is trivial and $|l^\times/l^{\times p}|=p$ for the field extension $l=k(\omega)$ with
$[l:k]=q$  \cite[Example 3.3]{L-Y}.
\end{example}

\begin{example} Take the situation of Example 2 (ii) and suppose that $r^2$ does not divide $q-1$:
Choose $H=k^\times/k^{\times r}=\{1,\xi,\dots,\xi^{r-1}\}$ maximal.
Any diagonal $H$-form $\varphi$ of degree $p$ over $k$ is isotropic and isomorphic to
$ m\times \langle 1,\xi,\dots,\xi^{p-1}\rangle$, so its dimension is a multiple of the degree $p$
and $d(\varphi)=1$.
   So here we have a very special situation: Every diagonal $H$-form is isotropic, of
    dimension a multiple of the degree, and has permanent 1.
The ideal generated by these diagonal $H$-forms contains all round universal diagonal forms over
$k=\mathbb{F}_q$, and these are all isotropic.
\end{example}

We conclude that the definition of Witt rings using $H$-equivalence given in \cite{H-P}  seems to give good results for
certain fields, if we only study  diagonal cubic forms.

 \section{Some observation on a Witt ring of separable forms and invariants of separable forms}\label{sec:5}

 Let $P^{sep}_d(k)$ denote the set of all isomorphism classes of nondegenerate separable forms of degree $d$ over $k$.
Define
$$W^{\rm sep}_d(k,H)=P^{\rm sep}_d(k)/\sim_H$$ and let
$$r^{\rm sep}_H:P^{\rm sep}_d(k)\rightarrow  W^{\rm sep}_d(k,H), \quad \Phi \rightarrow  [\Phi]$$
 denote the canonical map.
 $W^{\rm sep}_d(k,H)$ is a commutative semi-ring.

   A separable form $\Psi$ of degree $d$ over $k$ has an additive  inverse (which is again a separable form by construction)
  if and only if there are only finitely many
  elements $a_i\in H$ such that $a_i\Psi\not=\Psi$; i.e., if its stabilizer has finite index in $H$.
  Thus $W^{\rm sep}_d(k,H)$ is a ring, if $H$ is finite.

  \begin{example} (i) Let $H=\{1\}$ be trivial, then any form is an $H$-form and we obtain
   $$W_d(k,H)=W^{\rm sep}_d(k,H)=0.$$
   E.g., let $k=\mathbb{R}$. Then $|k^\times/k^{\times d}|=1$ for any odd integer $d=2l+1\geq 2$, hence
   $$W^{\rm sep}_{2l+1}(\mathbb{R},H)=0.$$
  \end{example}

 Let  $\widehat{W}^{\rm sep}_r(k)$ denote the set of separable forms in the Witt-Grothendiek ring
  $\widehat{W}_r(k)$. This set is a subring of the ring $\widehat{W}_r(k)$. 

\begin{lemma} (cf. \cite[9.1]{R}) Let $r\geq 3$. Let
$d: \widehat{W}^{\rm sep}_r(k)\rightarrow  k^\times/k^{\times r}$ be the permanent, and let $F\subset  \widehat{W}^{\rm sep}_r(k)$ be
 an ideal such that
${\rm dim}\,\varphi\equiv 0\,{\rm mod}\, r$ and $d(\varphi)=1$ for all $\varphi\in F$. Define $W_r(k,{\rm sep})=
\widehat{W}^{\rm sep}_r(k)/F$.
Let $I_r$ denote the kernel of the dimension index $dim: W^{\rm sep}_r(k)\rightarrow  \mathbb{Z}_r$. Then $d$ induces a surjective
morphism
$$d:I_r/I^2_r\rightarrow  k^\times/k^{\times r}.$$
\end{lemma}

 If $H$ is a finite subgroup of $k^\times/k^{\times r}$, then $P^{\rm sep}_r/\sim_H\cong W_r(k,{\rm sep})$
if the ideal $F$ is chosen to be the ideal in  $ W^{\rm sep}_r(k)$ generated by all separable $H$-forms.

Given the right choice of $r$ and $H$, is the above map is an isomorphism?
Ruppert showed that in case $r=p$ is a prime, and $k=\mathbb{F}_q$ is a finite field such that
 $q\equiv 1\,{\rm mod}\, p$ and $r^2$ does not divide $q-1$,
the group $I_p/I^2_p$ is not even finitely generated, if we choose the ideal $F$ generated by the form
$\langle 1,\xi,\dots,\xi^{p-1}\rangle$ ($\xi$ a primitive $p^{th}$ root of unity in $k$) \cite{R}.
 So one would need to cancel out a
 much bigger ideal in order to get an isomorphism here, if there exists one at all.
The reason this is so might be the fact that its generators
only consist of all the {\it diagonal} round universal forms, and no other separable ones, instead one should probably try this construction with the ideal
generated by forms of the kind $\langle 1,\xi,\dots,\xi^{p-1}\rangle$ and $\langle 1,\xi,\dots,\xi^{p-1}\rangle \otimes tr_{l/k}\langle c\rangle,$
with $c\in l^\times$ arbitrary.

\begin{example} (i) Let $k=\mathbb{C}$. Then $|k^\times/k^{\times d}|=1$ for any integer $d\geq 2$.
  Thus $H=\{1\}$ is trivial, any form over $\mathbb{C}$ is an $H$-form and
   $W^D_d(\mathbb{C},H)=0$, and $ W_d(\mathbb{C},H)=0.$
   (ii) Let $k=\mathbb{R}$. Then $|k^\times/k^{\times d}|=1$ for any odd integer $d=2l+1\geq 2$, hence
   $W^D_{2l+1}(\mathbb{R},H)=0$,    and  $ W_{2l+1}(\mathbb{R},H)=0.$
   For any even integer $d\geq 2$,
 $k^\times/k^{\times d}=\{1,-1\}$. If $d=2l$ is even  let $H=\{1,-1\}$,
  then any $H$-form is isomorphic to a form of the kind $m\times \langle1,-1\rangle$. For each diagonal form
  $\Psi$ of  degree $d$ over $k$ we obtain that $\Psi\sim_H m\times \langle 1\rangle$ or  $\Psi\sim_H n\times \langle -1\rangle$,
  for integers $m,n$, or that $\Psi\sim_H 0$. This yields
  $W^D_{2l}(\mathbb{R},H)=\mathbb{Z}.$
  \end{example}

\section{Different definitions of Witt rings using $I$-forms and $J$-forms}\label{sec:I-forms}

What happens if we use a different choice of ``hyperbolic forms'' in the definition of a Witt ring for forms of higher degree? It is certainly worth exploring this question.
We quickly sketch two possible choices, with obvious modifications for the corresponding Witt rings of diagonal forms.

\subsection{$I$-forms}

 Let $\varphi$ be a nondegenerate form of degree $d$ over $k$. Let $H= k^\times/k^{\times d}$. Here, $H$ can be finite or infinite.
 If $\varphi$ is an isotropic $H$-form 
 or, equivalently, if $\varphi$ is isotropic, round and universal, $\varphi$ is called an {\it $I$-form}.
 A form $\Phi$ is called {\it $I$-reduced} if $\Phi\cong \Psi\perp \varphi$ with an $I$-form $\varphi$
implies $\varphi=0$.

 We again obtain a kind of Witt-decomposition for forms of higher degree,
where here the role of the anisotropic kernel of the quadratic form is played by the $I$-reduced form:

\begin{theorem}
For each form $\Phi$ there is an $I$-reduced form $\Phi_I$ and an $I$-form $t_I(\Phi)$ such that
$$\Phi\cong \Phi_I \perp t_I(\Phi).$$
 For $d>2$, this decomposition is unique up to isomorphisms.
\end{theorem}

\begin{proof} For each form $\Phi$ there is an $H$-reduced form $\Phi_H$ and an $H$-form $t_H(\Phi)$ such that
$\Phi\cong \Phi_H \perp t_H(\Phi).$ For $d>2$, this decomposition is unique up to isomorphisms.
If the $H$-form $t_H(\Phi)$ is isotropic, then it is
 also an $I$-form, and the $H$-reduced form $\Phi_H$ is also $I$-reduced. We obtain $\Phi_I\cong\Phi_H$ and
$t_I(\Phi)\cong t_H(\Phi)$.
 If the $H$-form $t_H(\Phi)$ is anisotropic, then put $t_I(\Phi)=0$ and $\Phi_I=\Phi$ is already $I$-reduced.
 We obtain $\Phi_I\cong\Phi_H\perp t_H(\Phi)$ and
$t_I(\Phi)=0$.
For $d>2$, this decomposition is unique up to isomorphisms, since so is the decomposition into $H$-reduced forms
and $H$-forms.
\end{proof}

Two forms $\Phi$ and $\Psi$ are {\it $I$-equivalent}, written $\Phi \sim_I \Psi$, if there are $I$-forms
$\varphi_1$ and $\varphi_2$ such that $\Phi\perp \varphi_1\cong \Psi\perp \varphi_2$.
The equivalence relation $\sim_I$ is compatible with taking orthogonal sums and tensor products, thus the
 set of equivalence classes
$$W_d(k,I)=\{[\Phi]|\Phi \, \text{a form over}\, k\}$$
 is a commutative semiring.
More precisely, the surjective map
$r_I:P_d(k)\rightarrow W_d(k,H)$ preserves addition and multiplication, and $W_d(k,I)$ inherits the structure of
$P_d(k)$.

\begin{lemma} 
 Let $\Phi$ be a form over $k$ such
 that there are only finitely many
elements $a_i\in k^\times/k^{\times d}$ such that $a_i \Phi \not \cong \Phi$. Then
$\Phi$ has the additive inverse $\bar \Phi=\widetilde{\Phi}\perp \Phi\perp\widetilde{\Phi}$ in $W_d(k,I)$.
\end{lemma}

\begin{proof} Let $\Phi$ be a form over $k$,
such that there are only finitely many
elements $a_i\in k^\times/k^{\times d}$ such that $a_i \Phi \not \cong \Phi$.
The form $\Phi\perp \widetilde{\Phi}=\Phi\perp (\perp_{a_i}\Phi)$ is round and universal; i.e., an
$H$-form over $k$. 
Put $\bar \Phi=\widetilde{\Phi}\perp \Phi\perp\widetilde{\Phi}$, then
$\bar{\Phi}$ is the additive inverse of $\Phi$ in $W_d(k,I)$, since $\Phi\perp\bar \Phi=
\Phi\perp\widetilde{\Phi}\perp \Phi\perp\widetilde{\Phi}$ is an isotropic round universal form over $k$; i.e., an
$I$-form.
\end{proof}

Note that the additive inverse of a diagonal form again is a diagonal form.

\begin{corollary} Let $H= k^\times/k^{\times d}$ be finite. Then $W_d(k,I)$ is a commutative ring.
\end{corollary}

\begin{proof} Let $\Phi$ be a form over $k$,
then there can be only finitely many
elements $a_i\in k^\times/k^{\times d}$ such that $a_i \Phi \not \cong \Phi$.
 Then
$\Phi$ has the additive inverse $\bar \Phi$ in $W_d(k,H)$.
\end{proof}

 \begin{corollary}
 (i) For every $d>1$, each equivalence class $[\Phi]$
 in $W_d(k,I)$ contains at least one $I$-reduced representative $[\Phi]=[\Phi_I]$.\\
 (ii) For each $d>2$, the $I$-reduced representative in each equivalence class is uniquely determined, so $W_d(k,I)$
 can be viewed as subset of $P_d(k)$.
 \end{corollary}

 Let $r_I:P_d(k)\rightarrow W_d(k,I),$ $\Phi \mapsto [\Phi]$.
 If two forms are $I$-equivalent, then they also are $H$-equivalent, since by definition any $I$-form
is also an $H$-form. Thus the canonical semiring homomorphism
$${\rm can}: W_d(k,I)\rightarrow W_d(k,H), \quad [\Phi]\mapsto [\Phi]$$
 is surjective, since
 $r_H={\rm can}\circ r_I$, i.e $r_H:P_d(k)\rightarrow W_d(k,H)$ is the same as  $r_J:P_d(k)\rightarrow W_d(k,J)\to
 W_d(k,H)$.
 If $H= k^\times/k^{\times d}$ is finite, this is even a surjective ring homomorphism.

Note that
$W_2(k,I)$ is the classical Witt ring of quadratic forms: $I$-equivalence is the same as Witt equivalence, and
the $I$-forms over $k$ are exactly the hyperbolic quadratic forms over $k$. This makes this approach look like a canonical generalization of the classical Witt ring.

Another advantage is that we can  define trace maps in a sensible way.
This was possible for $W_d(k,H)$ only for certain groups $H$:
Let $l/k$ be a finite field extension and $s:l\rightarrow k$ a non-zero $k$-linear map.
For quadratic forms, $s_*$ induces a homomorphism of the classical Witt groups
$W(l)\rightarrow W(k),$
because given a hyperbolic quadratic form $q$ over $l$, the quadratic form $s_*q$ also is hyperbolic.

For $d> 2$, it is easy to check that given any $I$-form $\varphi$ over $l$, the form $s_*(\varphi)$ also is an $I$-form,
provided that the map $s:l\to k$ is surjective.
Thus, for surjective $s:l\to k$, $s_*$ induces a homomorphism of the semigroups
$$W_d(l,I)\rightarrow W_d(k,I).$$
Note that if $l$ is a separable field extension of $k$ and ${\rm char}\,k=0$, or if ${\rm char}\,k=p$ is prime
to $[l:k]$, then the trace $tr_{l/k}:l\rightarrow k$ is surjective. 

We do not see, however, that there is a canonical way to define a base change homomorphism in this set-up.

 \subsection{$J$-forms}

Quadratic hyperbolic forms over $k$  are quadratic forms which are generated
as orthogonal sum of forms
of the kind $\Phi\otimes \varphi$, where $\Phi$ is any quadratic form over $k$ and $\varphi$ the norm
 of a split composition algebra.  Looking at
 Schafer's classification of forms permitting composition \cite{S1}, already for $d=3$
 the forms permitting composition (even if we exclude the ones permitting Jordan composition here) have dimensions
 $1,2,3,5$ or $9$; i.e., the dimension is not a multiple of the degree.
 One way to rectify this is to only look at those forms
 with exactly that as an additional requirement, to be able to define something like an invariant induced by the dimension on the
 Witt ring later on. This motivates the following definition:

\begin{definition}  Let $\varphi$ be a nondegenerate form of degree $d$ over $k$. Then
$\varphi$ is called a {\it $J$-form}, if $\varphi$ is an $I$-form, which can be written as orthogonal sum of forms
of the kind $\Phi\otimes \varphi$, where $\Phi$ is any form and $\varphi$ an isotropic nondegenerate form which permits
composition of dimension not 5.
 A nondegenerate form $\Phi$ of degree $d$ over $k$
 is called {\it $J$-reduced} if $\Phi\cong \Psi\perp \varphi$
 with a $J$-form $\varphi$ implies $\varphi=0$.
\end{definition}

\begin{theorem}
For each form $\Phi$ there is a $J$-reduced form $\Phi_J$ and a $J$-form $t_J(\Phi)$ such that
$$\Phi\cong \Phi_J \perp t_J(\Phi).$$
 For $d>2$, this decomposition is unique up to isomorphisms.
\end{theorem}

\begin{proof} For each form $\Phi$ there is an $I$-reduced form $\Phi_I$ and an $I$-form $t_I(\Phi)$ such that
$\Phi\cong \Phi_I \perp t_I(\Phi).$ For $d>2$, this decomposition is unique up to isomorphisms.
If the $I$-form $t_I(\Phi)$ is  of the desired type; i.e., an orthogonal sum
of forms
of the kind $\Phi\otimes \varphi$, where $\Phi$ is any form and $\varphi$ an isotropic
nondegenerate form which permits either composition or Jordan composition, then it is
 also a $J$-form, and the $I$-reduced form $\Phi_I$ is also $J$-reduced. We obtain $\Phi_J\cong\Phi_I$ and
$t_J(\Phi)\cong t_I(\Phi)$.
 If the $I$-form $t_I(\Phi)$ is not of the above type, then put $t_J(\Phi)=0$ and $\Phi_J=\Phi$ is already $J$-reduced.
 We obtain $\Phi_J\cong\Phi_I\perp t_I(\Phi)$ and
$t_J(\Phi)=0$.
For $d>2$, this decomposition is unique up to isomorphisms, since so is the decomposition into $I$-reduced forms
and $I$-forms.
\end{proof}

Two forms $\Phi$ and $\Psi$ are {\it $J$-equivalent}, written $\Phi \sim_J \Psi$, if there are $J$-forms
$\varphi_1$ and $\varphi_2$ such that $\Phi\perp \varphi_1\cong \Psi\perp \varphi_2$.
The equivalence relation $\sim_J$ is compatible with taking orthogonal sums and tensor products, thus the
 set of equivalence classes
$$W_d(k,J)=\{[\Phi]|\Phi \, \text{a form over}\, k\}$$
 becomes a commutative semiring.
 However, here it seems impossible to construct the
additive inverse of a given element.

 \begin{corollary}
  (i) For every $d>1$, each equivalence class $[\Phi]$
 in $W_d(k,J)$ contains at least one $J$-reduced representative $[\Phi]=[\Phi_J]$.
 \\
 (ii) For each $d>2$, the $J$-reduced representative in each equivalence class is uniquely determined, so $W_d(k,J)$
 can be viewed as subset of $P_d(k)$.
 \end{corollary}

 Let $r_J:P_d(k)\rightarrow W_d(k,J),$ $ \Phi \mapsto [\Phi]$.
 Obviously, if two forms are $J$-equivalent, then they also are $I$-equivalent, since by definition any $J$-form
is also an $I$-form.  We obtain that the canonical semiring homomorphism
$${\rm can}: W_d(k,J)\rightarrow W_d(k,I), \quad [\Phi]\mapsto [\Phi]$$
 is surjective, since again
 $r_I={\rm can}\circ r_J$, i.e $r_I:P_d(k)\rightarrow W_d(k,I)$ is the same as  $r_J:P_d(k)\rightarrow W_d(k,J)\to
 W_d(k,I)$.

One advantage of the definition of $W_d(k,J)$ over $W_d(k,I)$ is that  we can define a canonical
base change homomorphism
$$r^* : W_d(k,J)\rightarrow W_d(l,J)$$
for each field extension $l$ over $k$, since $J$-forms stay $J$-forms under arbitrary field extensions.
 Moreover, again
$W_2(k,J)$ becomes the classical Witt ring of quadratic forms: $J$-equivalence
is the same as Witt equivalence here,
the quadratic $J$-forms over $k$ are exactly the hyperbolic quadratic forms over $k$.

Thus there are canonical surjective semigroup homomorphisms
$$W_d(k,J) \rightarrow W_d(k,I)\rightarrow W_d(k,H).$$

Let $l/k$ be a field extension.  For each surjective non-zero $s\in \Hom_k (l,k)$, the map
$$r^*:W_d(k,J)\rightarrow W_d(l,J), \quad \Phi\mapsto \Phi\otimes l$$
is a semiring homomorphism and
$$s_*: W_d(l,I)\rightarrow W_d(k,I),\quad \Gamma\mapsto s(\Gamma)$$
is a homomorphism.
The composite map
$$s_*\,{\rm can}\,r^*:W_d(k,J)\rightarrow W_d(l,J) \rightarrow W_d(l,I) \rightarrow  W_d(k,I)$$
is given by multiplication with $s_*\langle 1\rangle$, since we know that
$s_*(\Psi_l)\cong \Psi \otimes_k s_*(\langle 1\rangle).$

\begin{corollary}
 The image of $s_*$  in $W_d(k,I)$ is a set closed under multiplication
which is independent of the choice of the non-zero surjective $k$-linear map $s\in \Hom_k(l,k)$.
\end{corollary}

For the proof, see \cite[5.9, p.~48]{S}, it works here as well.

By the above result the kernel of $r^*$ must be contained in the annihilator of
$s_*\langle 1\rangle$.
(Assume $r^*(\Psi)\sim_J 0$; i.e., $r^*(\Psi)$ is a $J$-form, then it is also an $I$-form and we get $\Psi \otimes_k l$
is round isotropic universal, then
$\Psi \otimes s_*\langle 1\rangle \cong s_*(r^*(\Psi))$ is an $I$-form as well. I.e.,
$\Psi \otimes s_*\langle 1\rangle \sim_I 0$.)

Note, however, that unlike in the degree $2$ case, we cannot expect results as for instance given in  \cite[5.9, p.~49]{S},
since the form $s_*\langle 1\rangle$ is isotropic but indecomposable for $d>3$,
 and hence must be either an $I$-form (in this case the annihilator would be the whole ring and we would get no new information)
 or is $I$-reduced. In the latter case this would only imply that for each form $\Psi$ in the kernel of
 $r^*:W_d(k,J)\rightarrow W_d(l,J), \, \Phi\mapsto \Phi\otimes l$, that means for each form $\Psi$ where $\Psi\otimes_k l$
 becomes a direct sum of forms of the type $\Phi\otimes \varphi$, where $\Phi$ is any form and $\varphi$ an isotropic
nondegenerate form which permits composition, we know that
  $\Psi \otimes s_*\langle 1\rangle $ must be a round universal form.
\\\\
\emph{Acknowledgement:} Parts of this paper were written while the author visited the University of Ottawa, partially supported by the Centre de Recherches Math\'ematiques and Monica Nevins' NSERC Discovery Grant RGPIN-2020-05020. She would like to thank the Department of Mathematics and Statistics for its hospitality. The author would alos like V. Astier for his helpful comments and suggestions on a previous version of the paper.


\begin{thebibliography}{[B-C-R]}

\bibitem{H} D. K. Harrison,  {\it A Grothendieck ring of higher degree forms}. J. Algebra 35 (1975), 123-138.

\bibitem{H-P} D. K. Harrison, B. Pareigis, {\it Witt rings of
higher degree forms}, Comm. Alg. 16 (6) (1988), 1275-1313.

\bibitem{HLYZ}
H.-L. Huang, H. Lu, Y. Ye, C. Zhang, \emph{Diagonalizable higher degree forms and symmetric tensors.}
Linear Algebra Appl. 613 (2021), 151–169.

\bibitem{HLHZ} H.-L. Huang, H. Lu, Y. Huajun, C. Zhang,
\emph{On centres and direct sum decompositions of higher degree forms.}
Linear Multilinear Algebra 70 (22) (2022), 7290-7306.


\bibitem{K1} A. Keet, \emph{Decomposition of a higher degree form.} Comm. Algebra 30 (10) (2002), 4945–4963.

\bibitem{K2} A. Keet, \emph{Higher degree hyperbolic forms.} Quaest Math. 16 (1993), 413-442.


\bibitem{L-Y}
D. B. Leep, C. C. Yeomans,   {\it Quintic forms over $p$-adic fields}, J. Number Theory 57(2)(1996), 2031-2041.


\bibitem{OR1} M. O'Ryan, \emph{Products of higher degree forms}.
Linear Multilinear Algebra 63 (1) (2015), 34–45.

\bibitem{O} M. Orzech,  {\it Forms of low degree in finite fields}. Bull.
Austral. Math. Soc. 30(1) (1984), 45-58.

\bibitem{Pu1} S. Pumpl\"un,   {\it Round forms of higher degree}. Results in Mathematics 63 (1-2) (2008), 1-18.

\bibitem{Pu2} S. Pumpl\"un,   {\it Some classes of multiplicative forms of higher degree}.
Comm. Alg. 37 (2) (2009), 609-629.

\bibitem{Pu3}  S. Pumpl\"un, {\it Indecomposable forms of higher degree}.
 Math. Z. 253 (2) (2006), 347-360.

\bibitem{Pu5}  S. Pumpl\"un,  \emph{Diagonal forms of higher degree over function fields of $p$-adic curves.}
Int. J. Number Theory 16 (1) (2020),  161–172.


\bibitem{R} C. Rupprecht,  {\it Cohomological invariants for higher
degree forms}. PhD Thesis, Universit\"at Regensburg, 2003.

\bibitem{Sax} N. Saxena, \emph{On the centers of higher degree forms}. Unpublished article, 2005.
\\ \verb#www.math.uni-bonn.de/people/saxena/papers/laa05.pdf#


\bibitem{S1} R. D.~Schafer, ``An Introduction to Nonassociative Algebras.'' Dover Publ. Inc., New York, 1995.

\bibitem{S1} Schafer, R.D., {\it Forms permitting composition}, Adv.
Math. 4, 127-148.

\bibitem{S} W. Scharlau,  {\it``Quadratic and Hermitian Forms''}.
Springer-Verlag, Berlin-Heidelberg-New York, 1985.

\bibitem{SOR1} D. B. Shapiro, M. O'Ryan, \emph{Centers of higher degree trace forms.}
J. Pure Appl. Algebra 217 (12) (2013),  2263–2273.

\bibitem{SOR2} D. B. Shapiro, M. O'Ryan, \emph{Centers of higher degree forms.}
Linear Algebra Appl. 371 (2003), 301–314.

\end{thebibliography}
\end{document}